\documentclass[10pt]{amsart}
\usepackage{enumerate,amsmath,amssymb,latexsym,
amsfonts, amsthm, amscd}

\theoremstyle{plain}

\newtheorem{theorem}{Theorem}

\numberwithin{equation}{section}

\newcommand{\ra}{\rightarrow}

\newcommand{\R}{\mathbb{R}}
\newcommand{\rd}{{\rm d}}

\begin{document}

\title {Automorphisms of Quadratic Liouville Structures}

\date{}

\author[P.L. Robinson]{P.L. Robinson}

\address{Department of Mathematics \\ University of Florida \\ Gainesville FL 32611  USA }

\email[]{paulr@ufl.edu}

 \keywords{}

\begin{abstract}

We examine the diffeomorphisms of a symplectic vector space that preserve a chosen symplectic potential. Our examination yields an explicit description of these diffeomorphisms when the chosen potential differs from the canonical potential by the differential of a homogeneous quadratic in one of three broad classes. 

\end{abstract}

\maketitle

\bigbreak

\section{Introduction}

A {\it Liouville form} on the symplectic manifold $(M, \omega)$ is a one-form $\theta \in \Omega^1 (M)$ satisfying the equation $\rd \theta = \omega$. Such a one-form is also called a {\it symplectic potential}; its existence forces $M$ to be noncompact. The unique vector field $\zeta \in {\rm Vec} (M)$ with contraction $\zeta \lrcorner \; \omega = \theta$ is called the corresponding {\it Liouville field}; its time $t$ flow $\phi_t$ satisfies $\phi_t^* \omega = e^t\omega$. A diffeomorphism $g : M \ra M$ is called an automorphism of the Liouville structure $(M, \theta)$ precisely when $g^* \theta = \theta$ and a symplectomorphism of $(M, \omega)$ precisely when $g^* \omega = \omega$. We denote by ${\rm Aut} (M, \theta)$ the group comprising all automorphisms of $(M, \theta)$; it is a subgroup of the group ${\rm Sp} (M, \omega)$ comprising all symplectomorphisms of $(M, \omega)$. Each automorphism of $(M, \theta)$ preserves not only $\omega$ but also $\zeta$ and its flow $\phi$. 

\medbreak

We shall focus on the simplest of all symplectic manifolds. Let $(V, \Omega)$ be a symplectic vector space: thus, $V$ is a real vector space on which $\Omega$ is a nonsingular alternating bilinear form. When $V$ is given its ntural smooth manifold structure, the tangent space (of derivations) at each $z \in V$ is naturally isomorphic to the vector space $V$ itself: the canonical isomorphism $V \ra T_z V : v \mapsto v_z$ is defined by 
	\[v_z \psi = \psi'_z (v) = \frac{\rd}{\rd t} \psi(z + t v)\big|_0
\]
for each smooth function $\psi: V \ra \R$. These canonical isomorphisms transport $\Omega$ to a (translation-invariant) symplectic form $\omega$ on $V$ defined by 
	\[\omega_z (x_z, y_z) = \Omega (x, y) 
\]
whenever $x, y, z \in V$. The linear symplectic group ${\rm Sp} (V, \Omega)$ comprises all linear automorphisms $g$ of $V$ such that $\Omega (g x, g y) = \Omega (x, y)$ for all $x, y \in V$; naturally, this is a subgroup of the symplectomorphism group ${\rm Sp} (V, \omega)$. 

\medbreak 

This simplest symplectic manifold carries a unique ${\rm Sp}(V, \Omega)$-invariant Liouville form $\theta^0$: explicitly, 
	\[\theta^0_z (v_z) : = \frac{1}{2} \Omega (z, v)
\]
whenever $z, v \in V$. The verification that $\theta^0$ is a Liouville form for $(V, \omega)$ is entirely routine. Likewise routine is verification that $\theta^0$ is ${\rm Sp}(V, \Omega)$-invariant: if $g \in {\rm Sp} (V, \Omega)$ then $g^* \theta = \theta$. The canonical nature of $\theta^0$ makes it the only ${\rm Sp} (V, \Omega)$-invariant Liouville form on $(V, \omega)$. To see this, let $\theta$ be any ${\rm Sp} (V, \Omega)$-invariant Liouville form: as $\theta - \theta^0$ is closed and hence exact, $\theta = \theta^0 + \rd \psi$ for some smooth function $\psi: V \ra \R$. If $g \in {\rm Sp} (V, \Omega)$ then ${\rm d} (\psi \circ g) = {\rm d} (g^* \psi) = g^* ({\rm d} \psi) = {\rm d} \psi$ so that $\psi \circ g - \psi$ is constant; as $g$ fixes the origin, it follows that $\psi \circ g = \psi$. The action of ${\rm Sp} (V, \Omega)$ on $V \setminus \{0\}$ being transitive, we deduce that $\psi$ is constant and conclude that $\theta = \theta^0$.  

\medbreak 

Henceforth, we shall often suppress the superscript $0$ and write simply $\theta$ for the canonical Liouville form. Its corresponding Liouville field $\zeta$ is one-half the Euler field: its value at $z \in V$ is precisely $[\frac{1}{2} z]_z$; its Liouville flow $\phi_t$ at time $t$ sends $z$ to $e^{\frac{1}{2} t} z$. This canonical Liouville structure assumes a familiar form in terms of linear symplectic coordinates $(p_1, \dots , p_m, q_1, \dots , q_m)$: 
\[\theta = \frac{1}{2} \sum_{i = 1}^{m} (p_i {\rm d} q_i - q_i {\rm d} p_i),
\]
	\[\zeta = \frac{1}{2} \sum_{i = 1}^{m} (p_i \frac{\partial}{\partial p_i} + q_i \frac{\partial}{\partial q_i}). 
\]

\medbreak 

The automorphism group of this canonical Liouville structure is readily found. 

\begin{theorem} \label{can}

${\rm Aut} (V, \theta^0) = {\rm Sp} (V, \Omega)$

\end{theorem}

\begin{proof} 
The inclusion ${\rm Sp} (V, \Omega) \subseteq {\rm Aut} (V, \theta^0)$ is clear. For the reverse inclusion, let $g \in {\rm Aut} (V, \theta^0)$. Notice that $g$ is necessarily a symplectomorphism of $(V, \omega)$: each derivative $g'_z : V \ra V$ lies in ${\rm Sp} (V, \Omega)$. Notice also that $g$ preserves the Liouville field $\zeta$ and the Liouville flow $\phi$: thus 
\[g(e^{\frac{1}{2} t}z) = g(\phi_t z) = \phi_t g(z) = e^{\frac{1}{2} t}g(z).  
\]
Passage to the limit as $t \ra -\infty$ reveals that $g$ fixes the origin: $g(0) = 0$. Rearrangement then gives 
	\[g(z) = \frac{g(e^{\frac{1}{2} t}z) - g(0)}{e^{\frac{1}{2} t}}
\]
and further passage to the limit as $t \ra -\infty$ yields 
	\[g = g'_0 \in {\rm Sp} (V, \Omega). 
\]
\end{proof} 

\medbreak 

We remark that in the proof, $g$ need not be a diffeomorphism: any smooth map $g : V \ra V$ satisfying $g^* \theta = \theta$ must lie in ${\rm Sp} (V, \Omega)$. 

\medbreak 

Our aim in this paper is to examine the automorphism group of certain Liouville structures on $(V, \omega)$. As noted above, each Liouville form on $(V, \omega)$ equals $\theta + \rd \psi$ for some smooth function $\psi : V \ra \R$; we shall focus on those Liouville forms for which $\psi$ is a homogeneous quadratic. As we recall in Section 1, homogeneous quadratics $\psi : V \ra \R$ correspond to elements $A \in {\rm sp} (V, \Omega)$ of the symplectic Lie algebra, according to the conveniently normalized rule 
	\[\psi (z) = \frac{1}{4} \Omega (z, A z).
\]
We shall precisely determine ${\rm Aut} (V, \theta + \rd \psi)$ for the three classes of homogeneous quadratics $\psi$ corresponding to the conditions $A^2 = 0$, $A^2 = + I$ and $A^2 = -I$. Those $A$ such that $A^2 = -I$ are the compatible pseudounitary structures on $(V, \Omega)$; in such cases, ${\rm Aut} (V, \theta + \rd \psi)$ is precisely the corresponding pseudounitary group. Those $A$ such that $A^2 = +I$ correspond to Lagrangian splittings of $(V, \Omega)$; in such cases, ${\rm Aut} (V, \theta + \rd \psi)$ is effectively the full diffeomorphism group of the Lagrangian subspace ${\rm Ker} (A - I)$. Those $A$ such that $A^2 = 0$ behave more like the ones in the $A^2 = -I$ class, in that ${\rm Aut} (V, \theta + \rd \psi)$ is the centralizer of $A$ in the linear symplectic group ${\rm Sp} (V, \Omega)$; this is of course a generalization of Theorem \ref{can}.   

\medbreak

\section{General Considerations} 

Let $A: V \ra V$ be a linear map; its symplectic adjoint $A^{\dagger} : V \ra V$ is defined by the rule 
	\[\Omega (A^{\dagger} x, y) = \Omega (x, A y)
\]
for all $x, y \in V$. The symplectic Lie algebra of $(V, \Omega)$ is 
	\[{\rm sp} (V, \Omega) = \{ A: V \ra V : A^{\dagger} = - A \}. 
\]
Thus, ${\rm sp} (V, \Omega)$ comprises precisely all linear maps $A: V \ra V$ that preserve $\Omega$ in the infinitesimal sense 
	\[\Omega (Ax, y) + \Omega (x, Ay) = 0
\]
whenever $x, y \in V$; equivalently, ${\rm sp} (V, \Omega)$ comprises precisely all linear maps $A: V \ra V$ for which a symmetric bilinear form is defined by 
	\[V \times V \ra \R : (x, y) \mapsto \Omega (x,A y).
\] 
Note that each $A \in {\rm sp} (V, \Omega)$ exponentiates to a one-parameter subgroup $\{ e^{A t} : t \in \R \}$ of ${\rm Sp} (V, \Omega)$.  

\medbreak

Denote by $\mathcal{Q} (V)$ the space comprising all homogeneous quadratic real-valued functions on $V$. A linear isomorphism 
	\[{\rm sp} (V, \Omega) \ra \mathcal{Q} (V) : A \mapsto \psi^A
\]
is accordingly given by the rule 
	\[\psi^A (z) = \frac{1}{4} \Omega (z, A z)
\]
for all $z \in V$; here, the $\frac{1}{4}$ normalization is introduced for convenience in what follows. The differential $\rd \psi^A$ of $\psi^A$ is given by 
	\[\rd \psi^A_z (v_z) = \frac{1}{2} \Omega (z, A v) = - \frac{1}{2} \Omega (A z, v)
\]
for all $z, v \in V$. 

\medbreak 

Fix $A \in {\rm sp} (V, \Omega)$ and the associated quadratic $\psi^A \in \mathcal{Q} (V)$ as above. The corresponding Liouville form 
	\[\theta^A = \theta + \rd \psi^A
\]
is given by the formula
	\[\theta^A_z (v_z) = \frac{1}{2} \Omega (z - A z, v)
\]
and the corresponding Liouville field $\zeta^A$ by 
	\[\zeta^A_z = [\frac{1}{2} (z - A z)]_z
\]
for $z, v \in V$. The Liouville flow $\phi^A$ solves the differential equation 
	\[\frac{\rd z}{\rd t} = \frac{1}{2}(z - A z)
\]
so the time $t$ flow $\phi^A_t$ maps $z \in V$ to 
	\[\phi^A_t (z) = e^{\frac{1}{2} t} e^{- \frac{1}{2} A t} z = e^{\frac{1}{2} (I - A) t}. 
\]

Now, let $g : V \ra V$ be a diffeomorphism; for future use, we shall here assemble formulae for the effect of $g$ on the Liouville structure determined by $A$ as above. The pullback $g^* \theta^A$ of the Liouville form is  
	\[(g^* \theta^A)_z (v_z) = \theta^A_{gz} (g_* (v_z)) = \theta^A_{gz} ((g'_z v)_{gz})
\]	
so that explicitly
	\[(g^* \theta^A)_z (v_z) = \frac{1}{2} \Omega ( g(z) - A g(z), g'_z (v)) 
\]
and the condition $g^* \theta^A = \theta^A$ that $g$ preserve the Liouville form $\theta^A$ reads
	\[\Omega (g(z) - A g(z), g'_z (v)) = \Omega (z - A z, v).
\]
The pushforward $g_* \zeta^A$ of the Liouville field $\zeta^A$ is 
	\[g_* (\zeta^A_z) = g_*([\frac{1}{2} (z - A z)]_z)
\]
so that explicitly 
	\[g_* (\zeta^A_z) = [\frac{1}{2} (g'_z (z - A z))]_{g(z)} 
\]
and the condition $g_* \zeta^A = \zeta^A$ that $g$ preserve the Liouville field $\zeta^A$ reads
	\[g'_z (z - A z) = g(z) - A g(z). 
\]
Finally, the condition that $g$ preserve the Liouville flow $\phi^A$ reads 
	\[g(e^{\frac{1}{2} t} e^{- \frac{1}{2} A t} z) = e^{\frac{1}{2} t} e^{- \frac{1}{2} A t} g(z). 
\]

Let us write 
	\[{\rm Sp} (V, \Omega)^A = \{ g \in {\rm Sp} (V, \Omega) : g A = A g \} 
\]
for the centralizer of $A \in {\rm sp} (V, \Omega)$ in the linear symplectic group. 

\begin{theorem} \label{Z}
If $A \in {\rm sp} (V, \Omega)$ is arbitrary then 
	\[{\rm Sp} (V, \Omega)^A \subseteq {\rm Aut} (V, \theta^A). 
\]
\end{theorem} 

\begin{proof} 
If $g \in {\rm Sp} (V, \Omega)^A$ then each derivative $g'_z = g \in {\rm Sp} (V, \Omega)$ commutes with $A$ so that
	\[\Omega ( g(z) - A g(z), g'_z (v)) = \Omega ( g(z - Az), g(v)) = \Omega ( z - Az, v)
\]
and therefore 
	\[(g^* \theta^A)_z (v_z) = \theta^A_z (v_z) 
\]
which places $g$ in ${\rm Aut} (V, \theta^A)$. 
\end{proof} 

\medbreak 

In fact, ${\rm Sp} (V, \Omega)^A$ comprises precisely all those elements of ${\rm Aut} (V, \theta^A)$ whose action on $V$ is linear. To see this, let $g \in {\rm Aut} (V, \theta^A) \subseteq {\rm Sp} (V, \omega)$ be a linear map. The linearity of $g$ has two consequences: immediately, $g = g'_0 \in {\rm Sp} (V, \Omega)$; and the condition that $g$ preserve $\zeta^A$ yields 
\[g(z) - g(Az) = g(z - A z) = g'_z (z - A z) = g(z) - A g(z)
\]
from which $g A = A g$ follows by cancellation. As we have seen, it is possible for every automorphism of a Liouville structure on $(V, \omega)$ to be linear: this is true of the canonical Liouville structure, according to Theorem \ref{can}. As we shall see, $(V, \theta^A)$ can have nonlinear automorphisms: the inclusion of ${\rm Sp} (V, \Omega)^A$ in ${\rm Aut} (V, \theta^A)$ can be proper. 
\medbreak 

\section{The case $A^2 = -I$}

Let $A \in {\rm sp} (V, \Omega)$ satisfy $A^2 = -I$. At once, $A$ is a complex structure on the real vector space $V$; this enables us to regard $V$ as a complex vector space on which $i$ acts through $A$. Further, a complex inner product $\langle \cdot | \cdot \rangle^A$ is defined on $V$ (made complex via $A$) by the rule 
	\[\langle x | y \rangle^A = \Omega (x, A y) + i \Omega (x, y)
\]
for all $x, y \in V$; this inner product is always nonsingular but need not be positive definite. Such an $A$ is thus a compatible pseudounitary structure on $(V, \Omega)$ and is a compatible unitary structure in the positive definite case. The pseudounitary group ${\rm U} (V,\langle \cdot | \cdot \rangle^A)$ of $V$ for the inner product $\langle \cdot | \cdot \rangle^A$ is exactly the centralizer of $A$ in ${\rm Sp} (V, \Omega)$: 
	\[{\rm U} (V,\langle \cdot | \cdot \rangle^A) = {\rm Sp} (V, \Omega)^A. 
\]
 
\medbreak 

For convenience, in this section we shall often suppress $A$ from the notation, writing $i$ for its action on $V$ as a complex vector space and $\langle \cdot | \cdot \rangle$ for the inner product introduced above.  The homogeneous quadratic $\psi^A \in \mathcal{Q} (V)$ then assumes the form 
	\[\psi^A (z) = \frac{1}{4} \Omega (z, A z) = \frac{1}{4} \langle z | z \rangle
\]
and the time $2 t$ Liouville flow becomes
	\[\phi^A_{2 t} (z) = e^t e^{- i t} z. 
\]

\medbreak 

Now, let $g \in {\rm Aut} (V, \theta^A)$ be an automorphism of the Liouville structure $(V, \theta^A)$. It follows that $g \in {\rm Sp} (V, \omega)$ is a symplectomorphism, whence each derivative $g'_z$ lies in ${\rm Sp} (V, \Omega)$. As $g$ preserves the time $2 t$ Liouville flow,  
	\[g(e^{t} e^{- i t} z) = e^{t} e^{- i t} g(z)
\]
which upon replacement of $z$ by $e^{i t} z$ becomes 
	\[e^{- t} g(e^t z) = e^{ - i t} g(e^{ i t} z). 
\]
Let $t \ra - \infty$ in the first of these equations to see that $g(0) = 0$; use this to rewrite the second equation as 
	\[e^{ - t} \bigl(g(e^t z) - g(0) \bigr) = e^{ - i t} g(e^{ i t} z).
\]
Here, let $t \ra - \infty$ once more: the left-hand side converges to $g'_0 (z)$ but the right-hand side is periodic; we deduce that both sides are actually constant, with value $g'_0 (z)$. Return to the second equation and incorporate this new information: 
	\[e^{- t} g(e^t z) = g'_0 (z) = e^{ - i t} g(e^{ i t} z).
\]
Setting $t = 0$ on the left shows that 
	\[g = g'_0 \in {\rm Sp} (V, \Omega)
\]
while setting $t = \pi / 2$ on the right then shows that $g$ commutes with $i = A$. 

\medbreak 

We conclude that ${\rm Aut} (V, \theta^A)$ is contained in ${\rm Sp} (V, \Omega)^A$; as Theorem \ref{Z} gives the reverse inclusion, we have established the following result. 

\begin{theorem} \label{-I}
If $A \in {\rm sp} (V, \Omega)$ satisfies $A^2 = -I$ then 
	\[{\rm Aut} (V, \theta^A) = {\rm Sp} (V, \Omega)^A. 
\]
\end{theorem} 
\begin{flushright}
$\Box$
\end{flushright} 
\medbreak 
 
Otherwise said, ${\rm Aut} (V, \theta^A)$ coincides here with the pseudounitary group ${\rm U} (V,\langle \cdot | \cdot \rangle^A)$.	

\medbreak 

\section{The case $A^2 = + I$}

Let $A \in {\rm sp} (V, \Omega)$ satisfy $A^2 = +I$. In this case, $V$ has an $A$-eigendecomposition
	\[V = V^A_+ \oplus V^A_-
\]
where each of the spaces 
	\[V^A_+ = {\rm Ker} (A - I)
\]
and 
	\[V^A_- = {\rm Ker} (A + I)
\]
is a Lagrangian in $(V, \Omega)$ : that is, a maximal subspace of $V$ on which $\Omega$ is identically zero. In the opposite direction, if $V = V_+ \oplus V_-$ is a Lagrangian splitting of $(V, \Omega)$ then the linear map $A : V \ra V$ defined by $A| V_{\pm} = \pm I$ lies in ${\rm sp} (V, \Omega)$ and has square $+I$. 

\medbreak 

Again for convenience, in this section we shall often take the liberty of suppressing $A$ from the notation in this way; thus, the $A$-eigendecomposition $V = V_+ \oplus V_-$ engenders projections $P_{\pm} : V \ra V_{\pm}: v \mapsto v_{\pm} $. Note from $A^2 = I$ that 
	\[e^{ A t} = (\cosh t) I + (\sinh t) A
\]
so that 
	\[e^{A t} |V_{\pm} = e^{\pm t} I. 
\]
The homogeneous quadratic $\psi^A \in \mathcal{Q} (V)$ takes the form  
	\[\psi^A (z) = \frac{1}{4} \Omega (z, A z) = \frac{1}{2} \Omega (z_- , z_+ )
\]
while the time $t$ Liouville flow 
	\[\phi^A_{t} = e^{ \frac{1}{2} t} e^{ - \frac{1}{2} A t}
\]
has on $V_+$ the effect 
	\[\phi^A_{t} (z_+ ) = z_+ 
\]
and on $V_-$ the effect 
	\[\phi^A_{t} (z_- ) = e^t z_-. 
\]

\medbreak 

Now, let $g \in {\rm Aut} (V, \theta^A)$ be an automorphism of $(V, \theta^A)$; once again, each $g'_z$ lies in ${\rm Sp} (V, \Omega)$. As $g$ preserves the time $t$ Liouville flow, 
	\[g(z)_+ + e^t g(z)_- =\phi^A_t g(z) =  g( \phi^A_t z) = g(z_+ + e^t z_-)
\]
whence the limit as $t \ra - \infty$ yields 
	\[g(z)_+ = g(z_+). 
\]
This proves that $g$ maps $V_+$ to itself, thereby defining a function 
	\[f := g|V_+ : V_+ \ra V_+. 
\]
It proves more, namely that $g$ commutes with the projection $P_+$ of $V$ on $V_+$ along $V_-$:  
	\[P_+ \circ g = g \circ P_+ = f \circ P_+ 
\]
in consequence of which  
	\[P_+ \circ g'_z  = (P_+ \circ g)'_z = (f \circ P_+)'_z = f'_{z_+} \circ P_+.  
\]
As $g$ preserves the Liouville form, 
	\[2\Omega (z_-, v) = \Omega ((I - A) z, v) = \Omega ((I - A) g(z), g'_z v) = 2 \Omega (g(z)_-, g'_z v)
\]	
whence 
	\[\Omega (z_-, v) = \Omega (g(z)_-, P_+ (g'_z v))
\]
because $V_-$ is isotropic. From these facts, it follows that 
	\[\Omega (z_-, v) = \Omega (g(z)_-, \{ f'_{z_+} \circ P_+ \} v) = \Omega (\{ f'_{z_+} \circ P_+ \}^{\dagger}(g(z)_-), v)
\]
thus by nonsingularity of $\Omega$ 
	\[z_- = \{ f'_{z_+} \circ P_+ \}^{\dagger}(g(z)_-)
\]
and so 
	\[g(z)_- = \Bigl[\{ f'_{z_+} \circ P_+ \}^{\dagger} \Bigr]^{-1} (z_-). 
\]
The inverse here is justified as follows. Note that $g^{-1}$ is also an automorphism of $(V, \theta^A)$; accordingly, $g^{-1}$ also maps $V_+$ to itself and $f: V_+ \ra V_+$ is a diffeomorphism. Now consider the map 
	\[F_z : = f'_{z_+} \circ P_+ : V \ra V. 
\]
On the one hand, ${\rm Ran} (F_z) = V_+$ because $f$ is a diffeomorphism, so if $^{\perp}$ signifies $\Omega$-orthogonality then 
	\[{\rm Ker} (F_z^{\dagger}) = ({\rm Ran} F_z)^{\perp} = V_+^{\perp} = V_+
\]
because $V_+$ is Lagrangian. On the other hand, the rank-plus-nullity theorem now converts the evident inclusion $V_- \subseteq {\rm Ker} F_z$ to an equality, whence 
	\[{\rm Ran} (F_z^{\dagger}) = ({\rm Ker} F_z)^{\perp} = V_-^{\perp} = V_-
\]
because $V_-$ is Lagrangian. We conclude that $\{f'_{z_+} \circ P_+ \}^{\dagger}$ restricts to define an isomorphism from $V_-$ to itself. 

\medbreak 

The foregoing argument associates to each $g \in {\rm Aut} (V, \theta^A)$ in the automorphism group of $(V, \theta^A)$ an $f = g|V_+$ in the group ${\rm Diff} (V_+)$ comprising all diffeomorphisms of $V_+$; our formulae for $g(z)_+$ and $g(z)_-$ show that $g$ can be recovered from $f$. In the opposite direction, it is straightforward to check that if an $f \in {\rm Diff} (V_+)$ is given then the same formulae define a $g \in {\rm Aut} (V, \theta^A)$. In summary, we arrive at the following result. 

\begin{theorem} \label{+I} 
If $A \in {\rm sp} (V, \Omega)$ satisfies $A^2 = +I$ then an isomorphism 
	\[{\rm Diff} (V_+) \ra {\rm Aut} (V, \theta^A) : f \mapsto g 
\]
is defined by the rule 
	\[g(z) = f(z_+) + \Bigl[\{ f'_{z_+} \circ P_+ \}^{\dagger} \Bigr]^{-1} (z_-). 
\]
\end{theorem} 

In this case, the inclusion ${\rm Sp} (V, \Omega)^A  \subset {\rm Aut} (V, \theta^A)$ is decidedly proper: the Liouville structure $(V, \theta^A)$ has an abundance of nonlinear automorphisms.  
	
\medbreak 

\section{The case $A^2 = 0$} 

Let $A \in {\rm sp} (V, \Omega)$ satisfy $A^2 = 0$. The one-parameter subgroup $\{ e^{A t} : t \in \R \}$ of ${\rm Sp} (V, \Omega)$ is then given by 
	\[e^{A t} = I - A t 
\]
and the time $2 t$ Liouville flow by 
	\[\phi^A_{2 t} = e^t (I - A t). 
\]

\medbreak 

Again, let $g \in {\rm Aut} (V, \theta^A)$ be an automorphism of $(V, \theta^A)$; yet again, each $g'_z$ lies in ${\rm Sp} (V, \Omega)$. As $g$ respects the time $2 t$ Liouville flow, 
	\[g(e^t \{ z - A t z\}) = e^t \{ g(z) - A t g(z) \}
\]
whence passage to the limit as $t \ra - \infty$ yields 
	\[g(0) = 0. 
\]
Apply $e^{-t} \frac{\rm d}{{\rm d} t}$ and write $z_t = e^t \{ z - A t z\}$ to obtain 
	\[g'_{z_t} (z - (t + 1) A z) = g(z) - (t + 1) A g(z)
\]
or 
	\[g(z) - g'_{z_t} (z) = (t + 1) \{ A g(z) - g'_{z_t} (A z)  \} .
\]
A further passage to the limit as $t \ra - \infty$ yields $g'_{z_t} \ra g'_0$ so that $g'_{z_t} (z)  \ra g'_0 (z)$
and therefore
	\[A g(z) - g'_0 (A z) = \lim_{t \ra - \infty} (A g(z) - g'_{z_t} (A z)) = \lim_{t \ra - \infty} \frac{1}{t + 1} (g(z) - g'_{z_t} (z)) = 0 
\]
whence 
	\[g(z) - g'_{z_t} (z) = (t + 1) \{ g'_0 (A z) - g'_{z_t} (A z) \}. 
\]
Observe that the (smooth) derivative $g'$ is locally Lipschitz and let $t \ra - \infty$ yet once more: as $z_t$ converges to zero exponentially fast, so also does $g'_0 (A z) - g'_{z_t} (A z)$; we conclude from the last equation displayed above that 
	\[g(z) = g'_0 (z) 
\]
while the previous equation gave
	\[A g(z) = g'_0 (A z). 
\]
We deduce that $g = g'_0 \in {\rm} Sp (V, \Omega)$ is linear and that $A g = g A$. 

\medbreak 

Thus the inclusion reverse to that in Theorem \ref{Z} is established, to the following effect. 

\begin{theorem} \label{0}
If $A \in {\rm sp} (V, \Omega)$ satisfies $A^2 = 0$ then 
	\[{\rm Aut} (V, \theta^A) = {\rm Sp} (V, \Omega)^A. 
\]
\end{theorem} 
\begin{flushright}
$\Box$
\end{flushright} 
\medbreak

\section{Remarks}

The case $A^2 = +I$ is a familiar one in disguise. This choice of $A$ induces a symplectomorphism from $(V, \omega)$ to the cotangent bundle $T^* (V_+)$ of $V_+$ with its canonical symplectic form. In fact, for $z \in V$ define $\alpha(z) \in T^* _{z_+} (V_+)$ by requiring that if $v_+ \in V_+$ then 
	\[\alpha(z) [ (v_+)_{z_+} ] = \Omega (z_-, v_+ ). 
\]
The map $\alpha : V \ra T^* (V_+)$ is a vector bundle isomorphism from $V$ (with $P_+ : V \ra V_+$ as projection) to $T^* (V_+)$ (with its canonical projection) that not only pulls back the canonical symplectic form on $T^* (V_+)$ to $\omega$ but also pulls back the canonical (or `tautological') Liouville form on $T^* (V_+)$ to $\theta^A = \theta + \rd  \psi^A$ which is here given by 
	\[\theta^A_z (v_z) = \frac{1}{2} \Omega ((I - A) z, v) = \Omega (z_-, v_+). 
\]
In this way, we realize Theorem \ref{+I} as reflecting the fact that a diffeomorphism of a cotangent bundle $T^* Q$ preserves its canonical Liouville form precisely when it lifts a diffeomorphism of $Q$ (see page 186 of [1]); this has roots in the theory of Mathieu transformations (as discussed by Whittaker in his classic {\it A treatise on the analytical dynamics of particles and rigid bodies}). 

\medbreak 

The case $A^2 = 0$ covers the canonical Liouville structure, so Theorem \ref{0} is a substantial generalization of Theorem \ref{can}. The special subcase in which $A$ is the rank-one operator defined by fixing $a \in V$ and setting   
	\[A z = - \Omega (a, z) a
\]
so that 
	\[\psi^A (z) = \frac{1}{4} \Omega (a, z)^2
\] 
was considered in [2]. Indeed, there we fixed $a \in V$ and considered the homogeneous polynomials 
	\[\psi^a_n (z) = \frac{1}{2 n} \Omega (a, z)^n
\]
for all positive integers $n$. We found that the behaviour of the homogeneous quadratics differed from the behaviour of all the other homogeneous polynomials, which finding prompted the further consideration of homogeneous quadratics in the present paper.

\medbreak 

\begin{center} 
{\small R}{\footnotesize EFERENCES}
\end{center} 
\medbreak 

[1] R. Abraham and J.E. Marsden, {\it Foundations of Mechanics}, Second Edition, Benjamin Cummings (1978). 

\medbreak 

[2] P.L. Robinson, {\it Automorphisms of Liouville Structures}, University of Florida Preprint (2014). 

\medbreak

\end{document}